\documentclass[12pt,english,a4paper]{smfart}

\usepackage[utf8]{inputenc}
\usepackage[T1]{fontenc}
\usepackage[francais]{babel}
\usepackage{lmodern}
\usepackage{smfthm}
\usepackage[headings]{fullpage}
\usepackage{amssymb}

\let\cal\mathcal

\let\hat\widehat
\let\tilde\widetilde
\let\phi\varphi

\let\epsilon\varepsilon

\def\Q{{\bf Q}} 
\def\Z{{\bf Z}}
\def\C{{\bf C}}
\def\N{{\bf N}}
\def\R{{\bf R}}
\def\A{{\bf A}}
\def\E{{\bf E}}

\def\OO{{\cal O}}
\def\G{{\cal G}}

\def\Emb{{\mathrm{Emb}}}

\def\ra{{\longrightarrow}}

\def\Qp{{\Q_p}}
\def\Cp{{\C_p}}

\def\Qpbar{{\overline{\Q}_p}}

\def\At{{\tilde{\bf{A}}}}
\def\Atplus{{\tilde{\bf{A}}^+}}

\def\Et{{\tilde{\bf{E}}}}
\def\Etplus{{\tilde{\bf{E}}^+}}

\def\Gal{{\mathrm{Gal}}}

\author{Léo Poyeton}
\address{Institut Mathématique de Bordeaux \\
Université de Bordeaux
}
\email{leo.poyeton@math.u-bordeaux.fr}
\urladdr{https://www.math.u-bordeaux.fr/~lpoyeton/}

\date{\today}
\title{A criterion for Lubin's conjecture}

\begin{document}

\begin{abstract}
We prove that a formulation of a conjecture of Lubin regarding two power series commuting for the composition is equivalent to a criterion of checking that some extensions generated by the nonarchimedean dynamical system arising from the power series are Galois. As a consequence of this criterion, we obtain a proof of Lubin's conjecture in a new case.
\end{abstract}



\maketitle

\tableofcontents

\setlength{\baselineskip}{18pt}

\section*{Introduction}\label{intro} 
Let $K$ be a finite extension of $\Qp$, with ring of integers $\OO_K$ and maximal ideal $\mathfrak{m}_K$. Families of power series in $T\cdot \OO_K[\![T]\!]$ that commute under composition have been studied by Lubin \cite{lubin1994nonarchimedean} under the name of nonarchimedean dynamical systems, because of their interpretation as analytic transformations of the $p$-adic open unit disk. This study led Lubin to remark that ``experimental evidence seems to suggest that for an invertible
series to commute with a noninvertible series, there must be a formal group somehow in
the background''. 

Various results have been obtained to support Lubin's observation, see for instance the following non exhaustive list \cite{Li96}, \cite{Li97}, \cite{Li97b}, \cite{laubie2002systemes}, \cite{sarkis2005lifting}, \cite{sarkis10}, \cite{sarkis2013galois}, \cite{Ber17}, \cite{specter2018crystalline}, \cite{poyeton2022formal}.

This observation has lead to several versions of what might be called Lubin's conjecture, and these versions have all been proved under very strong assumptions on the nonarchimedean dynamical system considered.

In this note, we consider two power series $P, U \in T\cdot \OO_K[\![T]\!]$ such that $P \circ U = U \circ P$, with $P'(0) \in \mathfrak{m}_K$ and $U'(0) \in \OO_K^\times$, and we assume that $P(T) \neq 0 \mod \mathfrak{m}_K$ and that $U'(0)$ is not a root of unity. Our so called version of Lubin's conjecture is the following:

\begin{conj}
Let $P, U \in T\cdot \OO_K[\![T]\!]$ such that $P \circ U = U \circ P$, with $P'(0) \in \mathfrak{m}_K$ and $U'(0) \in \OO_K^\times$ not a root of unity, and such that $P(T) \neq 0 \mod \mathfrak{m}_K$. Then there exists a finite extension $E$ of $K$, a formal group $S$ defined over $\OO_E$, endomorphisms of this formal group $P_S$ and $U_S$ and a power series $h(T) \in T\OO_E[\![T]\!]$ such that $P \circ h = h \circ P_S$ and $U \circ h = h \circ U_S$.
\end{conj}

In the conjecture above, we say following Li's terminology \cite{Li97} that $P$ and $P_S$ are semiconjugate and that $h$ is an isogeny from $P_S$ to $P$. 

In several proven cases of this conjecture \cite{sarkis2005lifting,Ber17,specter2018crystalline}, the Lubin-Tate formal group is actually defined over $\OO_K$. However, this is not true in general.

The goal of this note is to prove the following theorem, which gives a new criterion to prove Lubin's conjecture in some cases:

\begin{theo}
\label{main theo}
Let $(P,U)$ be a couple of power series in $T\cdot \OO_K[\![T]\!]$ such that $P \circ U = U \circ P$, with $P'(0) \in \mathfrak{m}_K$ and $U'(0) \in \OO_K^\times$, and we assume that $P(T) \neq 0 \mod \mathfrak{m}_K$ and that $U'(0)$ is not a root of unity. Then there exists a finite extension $E$ of $K$, a Lubin-Tate formal group $S$ defined over $\OO_L$, where $E/L$ is a finite extension, endomorphisms of this formal group $P_S$ and $U_S$ over $\OO_E$, and a power series $h(T) \in T\OO_E[\![T]\!]$ such that $P \circ h = h \circ P_S$ and $U \circ h = h \circ U_S$ if and only if the following two conditions are satisfied:
\begin{enumerate}
\item there exists $V \in T\cdot\OO_K[\![T]\!]$, commuting with $P$, and an integer $d \geq 1$ such that $Q(T)=T^{p^d} \mod \mathfrak{m}_K$ where $Q=V \circ P$ ;
\item there exists a finite extension $E$ of $K$ and a sequence $(\alpha_n)_{n \in \N}$ where $\alpha_0 \neq 0$ is a root of $Q$ and $Q(\alpha_{n+1})=\alpha_n$ such that for all $n \geq 1$, the extension $E(\alpha_n)/E$ is Galois.
\end{enumerate}
\end{theo}

The proof relies mainly on the same tools and strategy used in \cite{poyeton2022formal}, which are the tools developed by Lubin in \cite{lubin1994nonarchimedean} to study $p$-adic dynamical systems, the ``canonical Cohen ring for norms fields'' of Cais and Davis \cite{cais2015canonical} and tools of $p$-adic Hodge theory following Berger's strategy in \cite{Ber14iterate}. 

As a corollary of our main theorem, we obtain the following result, which is a new instance of Lubin's conjecture:

\begin{theo}
\label{theo second}
Assume that $P(T) \in T\cdot\OO_K[\![T]\!]$ is such that $P(T) = T^p \mod \mathfrak{m}_K$ and that there exists $U \in T\cdot\OO_K[\![T]\!]$, commuting with $P$, such that $U'(0)$ is not a root of unity. Then there exists a finite extension $E$ of $K$, a Lubin-Tate formal group $S$ defined over $\OO_L$, where $E/L$ is a finite extension, endomorphisms of this formal group $P_S$ and $U_S$ over $\OO_E$, and a power series $h(T) \in T\OO_E[\![T]\!]$ such that $P \circ h = h \circ P_S$ and $U \circ h = h \circ U_S$.
\end{theo}

In order to prove our main theorem, we also need to prove that some extensions are strictly APF, which is a technical condition on the ramification of the extension. Cais and Davis have considered in \cite{cais2015canonical} what they called ``$\phi$-iterate'' extensions, and later on proved with Lubin that those extensions are strictly APF \cite{cais2014characterization}. Here, we show that that this result still holds for more general extensions which generalize the $\phi$-iterate extensions of Cais and Davis:

\begin{theo}
\label{theo extensions APF}
Let $K_\infty/K$ be an extension generated by a sequence $(u_n)$ of elements of $\Qpbar$ such that there exists a power series $P(T) \in T\cdot\OO_K[\![T]\!]$ with $P(T) = T^d$, where $d$ is a power of the cardinal of $k_K$, and an element $\pi_0$ of $\mathfrak{m}_K$ such that $u_0 = \pi_0$ and $P(u_{n+1}) = u_n$.

Then $K_\infty/K$ is strictly APF.
\end{theo}

\subsection*{Organization of the note}
The first section recalls the construction and properties of some rings of periods which are used in the rest of the paper. The second section is devoted to the proof of theorem \ref{theo extensions APF}, using the rings of periods of the first section in order to do so.  
In the third section we recall the main result of \cite{lubin1994nonarchimedean} which explains why ``Lubin's conjecture'' seems reasonable. In section 4, we prove that our version of Lubin's conjecture implies that the two conditions of theorem \ref{main theo} are satisfied. Section 5 and 6 show how to use $p$-adic Hodge theory, using the same strategy as in \cite{poyeton2022formal}, along with results from \cite{lubin1994nonarchimedean}, in order to prove that the infinite extension generated by such a $Q$-consistent sequence is actually generated by the torsion points of a formal Lubin-Tate group. In section 7, we show how to use the ``canonical Cohen ring for norms fields'' of Cais and Davis \cite{cais2015canonical} to prove that there is indeed an isogeny from an endomorphism of a formal Lubin-Tate group to $Q$. Section 8 is devoted to the proof of theorem \ref{theo second}.

\section{Rings of periods}
Let $K$ be a finite extension of $\Qp$, with uniformizer $\pi_K$, and let $K_0 = \Q_p^{\mathrm{unr}} \cap K$ denote the maximal unramified extension of $\Qp$ inside $K$. Let $q=p^h$ be the cardinality of $k_K$, the residue field of $K$, and let $e$ be the ramification index of $K$, so that $eh = [K:\Qp]$. Let $v_K$ denote the $p$-adic valuation on $K$ normalized so that $v_K(K^\times) = \Z$ and let $v_K$ still denote its extension to $\Qpbar$. Let $c > 0$ be such that $c \leq v_K(p)/(p-1)$. If $F$ is a subfield of $\C_p$, let $\mathfrak{a}_F^c$ be the set of elements of $F$ such that $v_K(x) \geq c$.

We now recall some definition of properties of some rings of periods which will be used afterwards. We refer mainly to \cite{cherbonnier1998representations} \cite{fontaine1994corps} for the properties stated here. The slight generalization to the classical rings by tensoring by $\OO_K$ over $\OO_{K_0}$ can for example be found in \cite{Ber14MultiLa}. 

Let $\OO_{\Cp}^\flat := \varprojlim\limits_{x \mapsto x^q}\OO_{\Cp}/\mathfrak{a}_{\Cp}^c$. This is the tilt of $\OO_{\Cp}$ and is perfect ring of characteristic $p$, whose fraction field $\Et$ is algebraically closed. It is endowed with a valuation $v_{\E}$ induced by the one on $K$. We let $W_K(\cdot)=\OO_K \otimes_{\OO_{K_0}}W(\cdot)$ denote the $\OO_K$-Witt vectors, and let $\Atplus = W_K(\Etplus)$ and $\At = W_K(\Et)$.  

Any element of $\At$ (resp. $\Atplus$) can be uniquely written as $\sum_{i \geq 0}\pi_K^k[x_i]$ with the $x_i \in \Et$ (resp. $\Etplus$). We let $w_k : \At \ra \R \cup \{+\infty\}$ defined by $w_k(x) = \inf_{i \leq k}v_{\E}(x_i)$.

For $r \in \R_+$, we let $\At^{\dagger,r}$ denote the subset of $\At$ of elements $x$ such that $w_k(x) + \frac{pr}{e(p-1)}k$ is $\geq 0$ for all $k$ and whose limits when $k \ra +\infty$ is $+\infty$. We let $n(r)$ be the smallest integer such that $r \leq p^{nh-1}(p-1)$.

We also let $\At = \bigcup_{r > 0}\At^{\dagger,r}$.  

\begin{lemm}
\label{lemm surconv inverse}
Let $x \in \At^{\dagger,r}+\pi_K^k\At$, then $\frac{x}{[\overline{x}]}$ is a unit of $\At^{\dagger,r'}+\pi_K^k\At$, with $r' = r+\frac{(p-1)e}{p}v_{\E}(\overline{x})$. 
\end{lemm}
\begin{proof}
Since $x \in \At^{\dagger,r}+\pi_K^k\At$, we can write $x = \sum_{i=0}^{k-1}\pi_K^i[x_i]$, where $x_0 = \overline{x}$, and $w_i(x) + \frac{pr}{e(p-1)}i \geq 0$ for all $i$ between $0$ and $k-1$. 

Now we can write $\frac{x}{[\overline{x}]} \in \At$ as $\sum_{i \geq 0}\pi_K^i[y_i]$, where $y_i = \frac{x_i}{\overline{x}}$ for $i$ between $0$ and $k-1$. In particular, $y_0 = 1$. Now a direct computation leads to the fact that $w_i(\frac{x}{[\overline{x}]}) + \frac{pr'}{e(p-1)}i \geq 0$ for all $i \leq k-1$, where $r' = r+\frac{(p-1)e}{p}v_{\E}(\overline{x})$. 

Using the fact that $\frac{x}{[\overline{x}]} \in (\At^{\dagger,r'}+\pi_K^k\At) \cap (1+\pi_K\At)$, we obtain that its inverse also lies into $\At^{\dagger,r'}+\pi_K^k\At$. 
\end{proof}

Let $\phi_q : \Etplus \to \Etplus$ denote the map $x \mapsto x^q$. This extends to a map $\Et \to \Et$ also given by $x \mapsto x^q$, and by functioriality of Witt vectors those maps extend into maps $\phi_q$ on $\Atplus$ and $\At$. 

Recall that there is a surjective map $\theta : \Atplus \to \OO_{\Cp}$ which is a morphism of rings. Moreover, if $x \in \Atplus$ and $\overline{x} = (x_n) \in \Etplus$, then $\theta \circ \phi_q^{-n}(x) = x_n \mod \mathfrak{a}_{\Cp}^c$. 

Also recall that, for $n \geq n(r)$, the maps $\theta \circ \phi_q^{-n}: \Atplus \to \OO_{\Cp}$ extend into surjective maps $\theta \circ \phi_q^{-n}: \At^{\dagger,r} \to \OO_{\Cp}$.

\section{Strictly APF extensions}
A theorem of Cais, Davis and Lubin \cite{cais2014characterization} gives a necessary and sufficient condition for an infinite algebraic extension $L/K$ to be strictly APF. In particular, this condition implies that what Cais and Davis have called a ``$\phi$-iterate'' extension in \cite{cais2015canonical} is strictly APF. 

Recall that a (slight generalization of what Cais and Davis in \cite{cais2015canonical} have called a) $\phi$-iterate extension $K_\infty/K$ is an extension generated by a sequence $(u_n)$ of elements of $\Qpbar$ such that there exists a power series $P(T) \in T\cdot\OO_K[\![T]\!]$ with $P(T) = T^d$, where $d$ is a power of the cardinal of $k_K$, and a uniformizer $\pi_0$ of $\OO_K$ such that $u_0 = \pi_0$ and $P(u_{n+1}) = u_n$. 

The main theorem of \cite{cais2014characterization} gives a necessary and sufficient condition for an infinite algebraic extension $L/K$ to be strictly APF, and in particular implies directly that those $\phi$-iterate extensions are strictly APF. 

In this section we will prove that this result remain true if we remove the assumption in the definition above that $\pi_0$ is a uniformizer of $\OO_K$, and instead just assume that $\pi_0 \in \mathfrak{m}_K$. We even allow $\pi_0$ to be equal to $0$, which is basically what we'll consider when looking at consistent sequences attached to a noninvertible stable power series.

If $L$ is a finite extension of $\Qp$, we let $v_L$ denote the $p$-adic valuation on $L$ normalized such that $v_L(L^\times) = \Z$, and we still denote by $v_L$ its extension to $\Qpbar$. If $L/M$ is a finite extension, we also let $\Emb_M(L,\Qpbar)$ denote the set of $M$-linear embeddings of $L$ into $\Qpbar$.

For the rest of this section, we let $P(T) \in T\cdot\OO_K[\![T]\!]$ with $P(T) = T^s$, where $s$ is a power of the cardinal of $k_K$, we let $\pi_0$ be any element of $\mathfrak{m}_K$, and we define a sequence $(v_n)_{n \in \N}$ of elements of $\Qpbar$ as follows: we let $v_0=\pi_0$, and for $n \geq 0$, we let $v_{n+1}$ be a root of $P(T)-v_n$.  We let $K_n = K(v_n)$ the field generated by $v_n$ over $K$, and we let $K_\infty = \bigcup_n K_n$. If $v_0 = 0$, then we choose $v_1$ to be $\neq 0$, so that the null sequence is excluded from our considerations.

\begin{prop}
\label{prop K(vn) tot ramified}
There exists $n_0 \geq 0$ and $d \geq 1$ such that, for all $n \geq n_0$, we have $v_{K_n}(v_n)=d$ and the extension $K_{n+1}/K_n$ is totally ramified of degree $s$. 
\end{prop}
\begin{proof}
The fact that the Weierstrass degree of $P$ is greater than $1$ along with Weierstrass preparation theorem show that the sequence $v_p(v_n)$ is strictly decreasing. In particular, there exists $n_0 \geq 0$ such that for $n \geq 0$, the Newton polygon of $P-v_n$ has only one slope, equal to $\frac{1}{s}v_p(v_n)$. This implies that for $n \geq n_0$, we have $v_p(v_{n+1}) = \frac{1}{s}v_p(v_n)$, and thus $v_{K_n}(v_{n+1}) = \frac{1}{s}v_{K_n}(v_n)$. 

Recall that, if $M/L/\Qp$ are finite extensions, then we have $[M:L]v_L \geq v_M$, with equality if and only if $M/L$ is totally ramified. Let $d_n:=v_{K_n}(v_n)$. Since $s$ is the degree of a non zero polynomial with coefficients in $K_n$ whose root is $v_{n+1}$, we know that $[K_{n+1}:K_n] \leq s$. This implies that $sv_{K_n} \geq [K_{n+1}:K_n]v_{K_n} \geq v_{K_{n+1}}$. For $n \geq n_0$, we have $d_n = s\cdot v_{K_n}(v_{n+1}) \geq [K_{n+1}:K_n]v_{K_n}(v_{n+1}) \geq v_{K_{n+1}}(v_{n+1}) = d_{n+1}$, so that the sequence $(d_n)_{n \in \N}$ is decreasing. Since this sequence takes its values in $\N$, it is stationary and therefore there exists $n_1 \geq n_0$ such that, for all $n \geq n_1$, $d_{n+1}=d_n$. In particular, this implies that the inequalities above are all equalities and thus that for $n \geq 1$, $s=[K_{n+1}:K_n]$ and that $K_{n+1}/K_n$ is totally ramified, and we can take $d = d_{n_1}$.
\end{proof}

Let us write $d = p^km$ where $m$ is prime to $p$. 

Since $P(T) = T^s \mod \mathfrak{m}_K$, the sequence $(v_n)$ gives rise to an element $\overline{v}$ of $\Etplus = \varprojlim\limits_{x \mapsto x^s}\OO_{\Cp}/\pi_0$. We let $\phi_s$ denote the $s$-power Frobenius map on $\Etplus$ and $\Atplus$. 

\begin{prop}
\label{prop lifts vbar}
There exists a unique $v \in \Atplus$ lifting $\overline{v}$ such that $\phi_s(v) = v$. Moreover, we have $\theta \circ \phi_s^{-n}(v) = v_n$.
\end{prop}
\begin{proof}
One can use the same argument as in \cite[Rem. 7.16]{cais2015canonical} to produce an element in $\Atplus$ such that $P(v) = \phi_s(v)$ and such that $\theta \circ \phi_s^{-n}(v) = v_n$ (note that one also needs to extend the results from ibid to the case where the Frobenius is replaced by a power of the Frobenius, which is straightforward).
 
Such an element automatically lifts $\overline{v}$ by definition of the theta map. For the uniqueness, one checks that the map $x \mapsto \phi_s^{-1}(P(x))$ is a contracting map on the set of elements of $\Atplus$ which lift $\overline{v}$, so that $v = \lim_{m \ra +\infty}\phi_s^{-m}(P^{\circ m}([\overline{v}]))$ and is thus unique. 
\end{proof}

Since $\Et$ is algebraically closed, there exists $\overline{u} \in \Et$ such that $\overline{u}^m = \overline{v}$. Since such a $\overline{u}$ necessarily has positive valuation, it actually belongs to $\Etplus$. 

Since $P(T) = T^s \mod \pi_0$, we can write $P(T) = T^s(1+\pi h(T))$, with $h(T) \in \frac{1}{T^s-1}\OO_K[\![T]\!]$. Let $Q(T) = T^s(1+\pi_0 h(T^m))^{1/m} \in \hat{\OO_K[\![T]\!][1/T]}$, which is well defined because $m$ is prime to $p$. Note that $Q(T)$ is overconvergent, meaning that it converges on some annulus bounded by the $p$-adic unit circle.

\begin{prop}
There exists $u \in \At^{\dagger}$, $u^m=v$. 
\end{prop}
\begin{proof}
We first construct $u$ such that $\phi_s(u) = Q(u)$. Just as the proof as in \ref{prop lifts vbar}, the map $x \mapsto \phi_s^{-1}(Q(x))$ is a contracting map on the set of elements of $\At$ lifting $\overline{u}$, so that $u = \lim_{m \ra +\infty}\phi_s^{-m}(Q^{\circ m}([\overline{u}]))$ and is unique. 

Therefore, there exists $u \in \At$ such that $\phi_s(u) = Q(u)$. Since $\overline{u} \in \Etplus$, we can write $u = [\overline{u}]+\pi_0z_1 \in \Atplus+\pi_0\At$. Let $r$ be such that $\frac{\pi_0}{[\overline{u}]^d} \in \At^{\dagger,r}$ and let $f=\frac{(p-1)e}{p}v_{\E}(\overline{x})$. Let us write $Q(T) = T^s(1+\frac{\pi_0}{T^s}g(T))^{1/m}$, with $g(T) \in \OO_K[\![T]\!]$. 

Now assume that there exists some $k \geq 1$ and $r' > 0$ such that $u \in \At^{\dagger,r'}+\pi_0^k$. We can thus write $u = u_k+\pi_0^kz_k$, where $u_k \in \At^{\dagger,r'}$ and $z_k \in \At$.  We have 

$$Q(u) = Q(u_k\pi_0^kz_k)=(u_k\pi_0^kz_k)^s(1+\frac{\pi_0}{(u_k\pi_0^kz_k)^s}g(u_k\pi_0^kz_k))^{1/m}.$$

Using the fact that $\frac{u}{[\overline{u}]}$ is a unit in $\At^{\dagger,r'+f}+\pi_0^k\At$, we obtain that $Q(u) \in \At^{\dagger,r''}+\pi_0^{k+1}\At$, where $r''=\max(s*r',r'+f)$. 

Since $\phi_s^{-1}(Q(u)) = u$, this implies that $u \in \At^{\dagger,r''/s}+\pi_0^{k+1}\At$.

By successive approximations, we have $u \in \At^{\dagger}$.

Finally, we compute $\phi_s(u^m) = \phi_s(u)^m = Q(u)^m = P(u^m)$ by construction of $Q$, so that $\phi_s(u^m) = P(u^m)$. Since $u^m$ lifts $\overline{u}^m = \overline{v}$, we have $u^m=v$ by unicity in proposition \ref{prop lifts vbar}.
\end{proof}

Recall that since $u \in \At^{\dagger}$, there exists some $r > 0$ such that $u \in \At^{\dagger,r}$ and there exists $n(r) \geq 0$ such that, for all $n \geq n(r)$, the element $u_n:=\theta \circ \phi_s^{-1}(u)$ is well defined and belongs to $\OO_{\Cp}$. Actually, since $u^m=v$, we have that $u_n^m=v_n$, and in particular we know that $v_K(u_n) \ra 0$. 

\begin{lemm}
\label{lemma false apf const1}
There exists a constant $c > 0$, independent of $n$, such that for any $n \geq n(r)$ and for any $g \in \G_{K_n}$ and any $i \geq 1$, we have
$$v_K(g(u_{n+i})-u_{n+i}) \geq c.$$
\end{lemm}
\begin{proof}
Let $n \geq n(r)$. We have $u_{n+i}^m=v_{n+i}$, so that $v_K(g(u_{n+i})^m-u_{n+i}^m) = v_K(g(v_{n+i}-v_{n+i})$. This means that
$$v_K(g(v_{n+i})-v_{n+i}) = v_K(g(u_{n+i})-u_{n+i})+(m-1)v_K(u_{n+i})$$
since $m$ is prime to $p$. 

Since $m$ is fixed and $v_K(u_{n}) \ra 0$, it suffices to prove that there exists $c > 0$ independent on $n$ such that $v_K(g(v_{n+i})-v_{n+i}) \geq c$ for all $g \in \G_{K_n}$. 

Since $P(T) = T^s \mod \mathfrak{m}_K$, and since $P^{\circ j}(v_{n+i})=v_n$, we already know that for all $n \geq 0$ and for all $g \in \G_{K_n}$, we have $v_K(g(v_{n+i}) - v_{n+i}) \geq 1$, so that $v_K(\frac{g(v_{n+i}}{v_{n+i}}-1) \geq 1 - v_K(v_{n+i}) \geq 1-v_K(v_n)$. The statement follows from the fact that $v_K(v_{n}) \ra 0$ when $n \ra +\infty$. 
\end{proof}

Recall that $d=p^km$, where $d$ is such that $v_{K_n}(v_n) = d$ for $n \gg 0$. Recall also that $s$ is a power of $p$, and let $j \geq 0$ be such that $s^j \geq p^k > s^{j-1}$. Let $f \geq 0$ be such that $p^{-f}s^j = p^k$. In particular, we have $v_{K_n}(u_{n+i}^{p^f}) = p^fs^{-j}v_{K_n}(u_n) = \frac{1}{mp^{k}}v_{K_n}(v_n)= \frac{d}{d} = 1$. 

We let $E_\infty = \bigcup_{n \geq 0}K(u_n)$, and $F = \Q_p^{\mathrm{unr}} \cap E_\infty$ be the maximal unramified extension of $\Qp$ inside $E_\infty$. Finally, we let $F^{(m)}$ denote the unramified extension of $F$ generated by the elements $[x^{1/m}]$, $x \in k_F$.  

For $n \geq n_0$, let $\pi_n$ denote a uniformizer of $\OO_{K_n}$. Since for all $n \geq n_0$ the extensions $K_{n+1}/K_n$ are totally ramified, the minimal polynomial of $\pi_{n+1}$ over $K_n$ is an Eisenstein polynomial, and we choose the $\pi_n$ so that $N_{K_{n+1}/K_n}(\pi_{n+1}) = \pi_n$ for all $n \geq n_0$.

\begin{lemm}
\label{lemm recurs apf}
For any $n \geq n(r)$, we can write $\pi_n = [h]\cdot u_{n+j}^{p^f}(1+x)$, with $x \in \OO_{K_{n+j}}$ and $h \in k_{F^{(m)}}$. 
\end{lemm}
\begin{proof}
Note that $v_{K_n}(\pi_n^m) = v_{K_n}(v_{n+j}^{p^f})$ and that both elements belong to $\OO_{K_{n+j}}$, so that we can write
$$\frac{\pi_n^m}{v_{n+j}^{p^f}} = [h_0]+\pi_{n+j}(\cdots),$$
with $h_0 \in k_F$. Taking the $m$-th root, this implies that there exists $h_1 \in k_{F^{(m)}}$ such that 
$$\frac{\pi_n}{u_{n+j}^{p^f}} = [h_1](1+\pi_{n+j}(\cdots)),$$ 
where the coefficients belong to $\OO_{K_{n+j}}$ and $h_1 \in k_{F^{(m)}}$. 
\end{proof}

\begin{theo}
\label{theo sAPF}
The extension $K_\infty/K$ is strictly APF.
\end{theo}
\begin{proof}
In order to prove the theorem, it suffices by \cite[Prop. 1.2.3]{Win83} to prove that the extension $F^{(m)}\cdot K_\infty/F^{(m)}\cdot K_{n_0}$ is strictly APF. 

To prove that $F^{(m)}\cdot K_\infty/F^{(m)}\cdot K$ is strictly APF, it suffices to prove that the $v_K$ valuations of the non constant and non leading coefficients of the Eisenstein polynomial of $\pi_{n+1}$ over $F^{(m)}\cdot K_n$, for $n \geq n_0$, are bounded below by a positive constant independent of $n$, so that $F^{(m)}\cdot K_\infty/F^{(m)}\cdot K_{n_0}$ satisfies the criterion of the main theorem (Thm 1.1) of \cite{cais2014characterization}. Let $n \geq n_0$.

By the lemma \ref{lemm recurs apf} and by induction, we can write 
$$\pi_{n+1} = u_{n+j+1}^{p^f}([h_0]+u_{n+1+2j}^{p^f}([h_1]+\cdots))$$
where the $h_i$ belong to $k_{F^{(m)}}$. 

Let $g \in \G_{F^{(m)}\cdot K_n}$. We have 
$$g(\pi_{n+1}) - \pi_{n+1} = g(u_{n+j+1}^{p^f})([h_0]) - u_{n+j+1}^{p^f}([h_0])+ \cdots$$
where all the terms on the RHS have $v_K$-valuation at least equal to $c > 0$ by lemma \ref{lemma false apf const1}, so that $v_K(g(\pi_{n+1}) - \pi_{n+1}) \geq c > 0$. 

The conjugates of $\pi_{n+1}$ over $K_n$ are the elements $g(\pi_{n+1})$, for $g \in \G_{K_n}$, and satisfy the conditions $v_K(g(\pi_{n+1}) - \pi_{n+1}) \geq c > 0$, which ensures that the $v_K$ valuations of the non constant and non leading coefficients of the Eisenstein polynomial of $\pi_{n+1}$ over $F^{(m)}\cdot K_n$ are bounded below by a positive constant independent of $n$, which is what we wanted.
\end{proof}

\section{Non archimedean dynamical systems}
\label{section def situation}
Let $K$ be a finite extension of $\Qp$, with ring of integers $\OO_K$, uniformizer $\pi$, maximal ideal $\mathfrak{m}_K$ and residual field $k$ of cardinal $q=p^h$. We let $K_0 = K \cap \Q_p^{\mathrm{nr}}$ be the maximal unramified extension of $\Qp$ inside $K$ and we let $\OO_{K_0}$ denote its ring of integers. We let $\Cp$ denote the $p$-adic completion of $\Qpbar$. Let $P, U \in T\cdot \OO_K[\![T]\!]$ such that $P \circ U = U \circ P$, with $P'(0) \in \mathfrak{m}_K$ and $U'(0) \in \OO_K^\times$. In this note, we assume that the situation is ``interesting'', namely that $P(T) \neq 0 \mod \mathfrak{m}_K$ and that $U'(0)$ is not a root of unity. 

\begin{prop}
\label{lubinmaintheo}
There exists a power series $H(T) \in T\cdot k[\![T]\!]$ and an integer $d \geq 1$  such that $H'(0) \in k^\times$ and $P(T) = H(T^{p^d}) \mod \mathfrak{m}_K$.
\end{prop}
\begin{proof}
This is theorem 6.3 and corollary 6.2.1 of \cite{lubin1994nonarchimedean}. 
\end{proof}

Near the end of his paper \cite{lubin1994nonarchimedean}, Lubin remarked that ``Experimental evidence seems to suggest that for an invertible series to commute with a noninvertible series, there must be a formal group somehow in the background.'' This has led some authors to prove some cases (see for instance \cite{Li96}, \cite{Li97}, \cite{Li97b}, \cite{sarkis2005lifting}, \cite{sarkis10}, \cite{sarkis2013galois}, \cite{Ber17}, \cite{specter2018crystalline}) of this ``conjecture'' of Lubin. The various results obtained in this direction can be thought of as cases of the following conjecture:

\begin{conj}
\label{lubinconj}
Let $P, U \in T\cdot \OO_K[\![T]\!]$ such that $P \circ U = U \circ P$, with $P'(0) \in \mathfrak{m}_K$ and $U'(0) \in \OO_K^\times$ not a root of unity, and such that $P(T) \neq 0 \mod \mathfrak{m}_K$. Then there exists a finite extension $E$ of $K$, a formal group $S$ defined over $\OO_E$, endomorphisms of this formal group $P_S$ and $U_S$, and a power series $h(T) \in T\cdot\OO_E[\![T]\!]$ such that $P \circ h = h \circ P_S$ and $U \circ h = h \circ U_S$.
\end{conj}

\begin{rema}
While in many instances of the cases where this conjecture is proven, the formal group is actually defined over $\OO_K$ \cite{sarkis2005lifting,Ber17,specter2018crystalline}, one can produce instances where the formal group is defined over the ring of integers of a finite unramified extension of $\OO_K$ \cite[\S 3]{berger2019nonarchimedean}. The author does not know of a case where the extension $E$ the formal group is defined over is ramified over $K$ so it might be possible that the assumption that $E$ is an unramified extension of $K$ can be enforced.
\end{rema}

\section{Endomorphisms of a formal Lubin-Tate group}
Let $P, U \in T\cdot \OO_K[\![T]\!]$ such that $P \circ U = U \circ P$, with $P'(0) \in \mathfrak{m}_K$ and $U'(0) \in \OO_K^\times$ not a root of unity, and such that $P(T) \neq 0 \mod \mathfrak{m}_K$. In this section, we assume that there exists a finite extension $E$ of $K$, a Lubin-Tate formal group $S$ defined over $\OO_L$ with $E/L/K$ finite, a power series $h \in T\cdot\OO_E[\![T]\!]$ and an endomorphism $P_S$ of $S$ such that $h$ is an isogeny from $P_S$ to $P$. 

\begin{lemm}
\label{lemma renormPintoQ}
There exists $V \in T\cdot\OO_K[\![T]\!]$, commuting with $P$, and an integer $d \geq 1$ such that $Q(T)=T^{p^d} \mod \mathfrak{m}_K$ where $Q=V \circ P$. Moreover, there exists $Q_S$ endomorphism of $S$ such that $h$ is an isogeny from $Q_S$ to $Q$.
\end{lemm}
\begin{proof}
First note that for any $V_S$ invertible series commuting with $P_S$, there corresponds an invertible power series $V$ commuting with $P$. Since $S$ is a formal Lubin-Tate group over $\OO_L$, $P_S$ corresponds to the multiplication by an element $\alpha \in \mathfrak{m}_L$. Let $[\varpi_L]$ denote the multiplication by $\varpi_L$ on $S$, a uniformizer of $\OO_L$ such that $[\varpi_L](T) = T^{\mathrm{Card}(k_L)} \mod \mathfrak{m}_L$ (we can find such a uniformizer since $S$ is a Lubin-Tate formal group defined over $\OO_L$). Since $\alpha \in \mathfrak{m}_L$, there exists $c \in \OO_L^\times$ and an integer $d \geq 1$ such that $\alpha = c \cdot \varpi_L^d$. In particular, we have $\mathrm{wideg([\alpha])} = \mathrm{wideg(P)} = \mathrm{wideg([\varpi_L^d])} = \mathrm{Card}(k_L)^d$.

We let $V$ denote the power series commuting with $P$ such that $h \circ [c^{-1}] = V \circ h$. We then have that $h \circ [c^{-1}] \circ [\alpha] = V \circ P \circ h$, and that $h \circ [c^{-1}] \circ [\alpha] = h \circ [\varpi_L^d]$, so that $h$ is an isogeny from $[\varpi_L^d]$ to $Q:=V \circ P$. Reducing modulo $\mathfrak{m}_L$, we get that 
$$h(T)^{\mathrm{Card}(k_L)^d} = h(T^{\mathrm{Card}(k_L)^d}) = h \circ Q \mod \mathfrak{m}_L$$
so that $Q = T^{\mathrm{Card}(k_L)^d} = T^{\mathrm{wideg(P)}} \mod \mathfrak{m}_L$.
\end{proof}

Let $(u_n)_{n \in \N}$ be a sequence of elements of $\Qpbar$ such that $u_0 \neq 0$ is a root of $Q_S$, and $Q_S(u_{n+1})=u_n$. In Lubin's terminology (see the definition on page 329 of \cite{lubin1994nonarchimedean}), the sequence $(v_n)$ is called a $Q_S$-consistent sequence. Let $E_n=E(u_n)$ and let $E_\infty = \bigcup_n E_n$. Then for all $n \geq 1$, the extensions $E_n/E$ are Galois. 

Let $Q$ as in lemma \ref{lemma renormPintoQ} and let $v_n:=h(u_n)$.

\begin{lemm}
\label{lemma endo Lubin implies Q-consistent}
The sequence $(v_n)_{n \in \N}$ is $Q$-consistent, and the extensions $E(v_n)/E$ are Galois for all $n \geq 1$.
\end{lemm}
\begin{proof}
We know that $E_n/E$ are Galois abelian extensions. Since $E \subset E(v_n) \subset E_n$, this implies that the extensions $E(v_n)/E$ are Galois. The fact that the sequence $(v_n)_{n \in \N}$ is $Q$-consistent follows directly from the fact that $h$ is an isogeny from $Q_S$ to $Q$. 
\end{proof}

\section{Embeddings into rings of periods}
Let $L:=K_{n_0}$ with $n_0$ as in proposition \ref{prop K(vn) tot ramified}. Since $P(T) = T^{p^d} \mod \mathfrak{m}_K$, there exists $m \geq 1$ such that $P^{\circ m}$ acts trivially on $k_L$, so that the degree $r$ of $Q$ is a power of the cardinal of $k_L$, and we let $Q:= P^{\circ m}$ after having chosen such an $m$. We let $w_0 = v_{n_0}$ and $(w_n)$ be a sequence extracted from $(v_n)$ such that $Q(w_{n+1})=w_n$. Let $L' = \Q_p^{\mathrm{unr}} \cap L$ be the maximal unramified extension of $\Qp$ inside $L$, and let $\Atplus := \OO_L \otimes_{\OO_{L'}}W(\Etplus)$. 

Since $K_\infty/L$ is strictly APF, there exists by \cite[4.2.2.1]{Win83} a constant $c = c(K_{\infty}/L) > 0$ such that for all $F \subset F'$ finite subextensions of $K_\infty/L$, and for all $x \in \mathcal{O}_{F'}$, we have 
$$v_L(\frac{N_{F'/F}(x)}{x^{[F':F]}}-1) \geq c.$$
We can always assume that $c \leq v_L(p)/(p-1)$ and we do so in what follows. By §2.1 and §4.2 of \cite{Win83}, there is a canonical $\G_L$-equivariant embedding $\iota_L : A_L(K_\infty) \hookrightarrow \Etplus$, where $A_L(K_\infty)$ is the ring of integers of $X_L(K_\infty)$, the field of norms of $K_\infty/L$. We can extend this embedding into a $\G_L$-equivariant embedding $X_L(K_\infty) \hookrightarrow \Et$, and we note $\E_K$ its image.

It will also be convenient to have the following interpretation for $\Etplus$: 
$$\tilde{\bf{E}}^+ = \varprojlim\limits_{x \to x^p} \mathcal{O}_{\C_p} = \{(x^{(0)},x^{(1)},\dots) \in \mathcal{O}_{\C_p}^{\N}~: (x^{(n+1)})^p=x^{(n)}\}.$$
To see that this definition coincides with the one given in \S 1, we refer to \cite[Prop. 4.3.1]{BrinonConrad}.

Note that, even though $\E_K$ depends on $K_\infty$ rather than on $L$, it is still sensitive to $L$:
\begin{prop}
\label{extension corps des normes bas depend de K}
Let $K'$ be a finite extension of $L$ contained in $K_\infty$. Let $K_1$ (resp. $K_1'$) be the maximal tamely ramified extension of $K_\infty/L$ (resp. $K_\infty/K'$). Then as subfields of $\Et$, $\E_{K'}$ is a purely inseparable extension of $\E_K$ of degree $[K_1':K_1]$. In particular, $\E_{K_1}=\E_K$.
\end{prop}
\begin{proof}
See \cite[Prop. 4.14]{cais2015canonical}.
\end{proof}

The sequence $(w_n)$ defines an element $\overline{w} \in \Etplus$.

\begin{prop}
\label{prop lifts wbar}
There exists a unique $w \in \Atplus$ lifting $\overline{w}$ such that $Q(w) = \phi_r(w)$. Moreover, we have that $\theta \circ \phi_r^{-n}(w) = w_n$. 
\end{prop}
\begin{proof}
This is the same proof as for the proposition \ref{prop lifts vbar}. 
\end{proof}

For all $k \geq 0$, we let 
$$R_k:=\{ x \in \Atplus, \theta \circ \phi_d^{-n}(x) \in \mathcal{O}_{L_{n+k}} \textrm{ for all } n \geq 1 \}.$$

\begin{prop}
\label{existsz_k}
For all $k \geq 0$, there exists $z_k \in R_k$ such that $R_k = \OO_L[\![z_k]\!]$.
\end{prop}
\begin{proof}
Note that for all $k \geq 0$, $R_k$ is an $\OO_L$-algebra, separated and complete for the $\pi_L$-aidc topology, where $\pi_L$ is a uniformizer of $\OO_L$. If $x \in R_k$, then its image in $\Etplus$ belongs to $\lim\limits_{x \mapsto x^r}\OO_{L_{n+k}}/\mathfrak{a}_{L_{n+k}}^c$.  

Note that the natural map $R_k/\pi_LR_k \to \Etplus$ is injective. To prove this, we need to prove that $\pi_L\Atplus \cap R_k = \pi_LR_k$. Let $x \in R_k \cap \pi_L\Atplus$ and let $y \in \Atplus$ be such that $x=\pi_Ly$. Then since $x \in R_k$ we have that $\theta \circ \phi_r^{-n}(x) \in \mathcal{O}_{L_{n+k}}$ and thus $\theta \circ \phi_r^{-n}(y) \in \frac{1}{\pi_L}\mathcal{O}_{L_{n+k}}$. But since $\theta \circ \phi_r^{-n}$ maps $\Atplus$ into $\OO_{\Cp}$ we get that $\theta \circ \phi_r^{-n}(y) \in L_{n+k} \cap \OO_{\Cp} = \OO_{L_{n+k}}$. Therefore the natural map $R_k/\pi_LR_k \to \Etplus$ is injective.

We know by the theory of field of norms that $\lim\limits_{x \mapsto x^r}\OO_{L_n}/\mathfrak{a}_{L_n}^c \simeq k_L[\![\overline{v}]\!]$ for some $\overline{v} \in \Etplus$, so that the valuation induced by $v_L$ on $\Etplus$ is discrete on $R/\pi_LR$.
Let $\overline{u} \in R/\pi_LR$ be an element of minimal valuation within
$$\{x \in R/\pi_LR, v_L(x) > 0\}.$$

Since the valuation on $R/\pi_LR$ is discrete, and since this set is nonempty because it contains the image of the element $w$ given by proposition \ref{prop lifts wbar}, such an element $\overline{u}$ exists, and we have $R/\pi_LR = k_L[\![\overline{u}]\!]$, so that $R=\OO_L[\![u]\!]$ for $u \in R$ lifting $\overline{u}$ since $R$ is separated and complete for the $\pi_L$-adic topology.
\end{proof}

\begin{prop}
\label{existz}
There exists $k_0 \geq 0$ such that, for all $k \geq k_0$, we can take $z_{k+1} = \phi_r^{-1}(z_k)$ and we let $z=z_{k_0}$.
\end{prop}
\begin{proof}
The proof of proposition \ref{existsz_k} shows that $R_k/\pi_LR_k$ injects into $\varprojlim\limits_{x \mapsto x^r}\mathcal{O}_{L_{n+k}}/\mathfrak{a}_{L_{n+k}}^c$. By \cite[Prop. 4.2.1]{Win83} $\varprojlim\limits_{x \mapsto x^r}\mathcal{O}_{L_{n+k}}/\mathfrak{a}_{L_{n+k}}^c$ is the image of ring of integers of the field of norms of $L_\infty/L_k$ inside $\Et$ by the embedding $\iota_L$, and we will denote $\varprojlim\limits_{x \mapsto x^r}\mathcal{O}_{L_{n+k}}/\mathfrak{a}_{L_{n+k}}^c$ by $Y_k$. We normalize the valuation of $Y_k$ so that $v_{Y_k}(Y_k)=\Z$. By proposition \ref{extension corps des normes bas depend de K}, we get that for $k \geq n_0$, we have $Y_{k+1} = \phi_r^{-1}(Y_k)$ and thus the valuation $v_{Y_{k+1}}$ is equal to $rv_{Y_k}$. 

Now let $v(k):=v_{Y_k}(\overline{z_k})$ for $k \geq 0$. We know by definition of the sets $R_k$ that $\phi_r^{-1}(z_k) \in R_{k+1}$ for all $k \geq 1$ and thus $v_{Y_{k+1}}(\overline{z_{k+1}}) \leq rv_{Y_k}(\phi_r^{-1}(\overline{z_k}))$ by construction of the $z_k$. This implies that the sequence $(v(k))_{k \geq n_0}$ is nonincreasing, and since it is bounded below by $1$, this implies that there exists some $k_0 \geq n_0$ such that, for all $k \geq k_0$, we have $v(k) = v(k_0) > 0$. Thus for all $k \geq k_0$ we have $v_{Y_{k+1}}(\overline{z_{k+1}}) = v_{Y_k}(\overline{z_k})$ and by construction of the $z_k$ this implies that we can take $z_{k+1} = \phi_r^{-1}(z_k)$ which concludes the proof.  
\end{proof}

We now let $k_0$ be as in proposition \ref{existz}. Note that in particular, for all $k \geq k_0$, we have $R_k = \phi_r^{k_0-k}(\OO_E[\![w]\!]) = \OO_E[\![\phi_r^{k_0-k}(w)]\!]$.

\begin{lemm}
\label{ajuster z}
The ring $\mathcal{O}_L[\![z]\!]$ is stable by $\phi_r$. Moreover, there exists $a \in \mathfrak{m}_L$ such that if $z'=z-a$ then there exists $S(T) \in T\cdot \mathcal{O}_L[\![T]\!]$ such that $S(z') = \phi_r(z')$ and $S(T) \equiv T^r \mod \mathfrak{m}_L$.
\end{lemm}
\begin{proof}
The set
$$\left\{x \in \Atplus, \theta \circ \phi_r^{-n}(x) \in \mathcal{O}_{L_{n+k_0}} \textrm{ for all } n \geq 1 \right\}$$
is clearly stable by $\phi_r$ and equal to $\mathcal{O}_L[\![z]\!]$ by proposition \ref{existz}, so that $\phi_r(z) \in \mathcal{O}_L[\![z]\!]$ and so there exists $R \in \mathcal{O}_L[\![T]\!]$ such that $R(z) = \phi_r(z)$. In particular, we have $\overline{R}(\overline{z}) = \overline{z}^r$ and so $R(T) \equiv T^r \mod \mathfrak{m}_L$.

Now let $\tilde{R}(T) = R(T+a)$ with $a \in \mathfrak{m}_L$ and let $z'=z-a$. Then $\phi_r(z')=\phi_r(z-a)=R(z)-a = \tilde{R}(z')-a$ and we let $S(T) = \tilde{R}(T)-a$ so that $\phi_r(z')=S(z')$. For $S(0)$ to be $0$, it suffices to find $a \in \mathfrak{m}_L$ such that $R(a)=a$. Such an $a$ exists since we have $R(T) \equiv T^r \mod \mathfrak{m}_L$ so that the Newton polygon of $R(T)-T$ starts with a segment of length $1$ and of slope $-v_p(R(0))$. 

Now, we have $S(z')=\phi_r(z')$ and so $\overline{S}(\overline{z'}) = \overline{z'}^r$, so that $S(T) \equiv T^r \mod \mathfrak{m}_L$.
\end{proof}

Lemma \ref{ajuster z} shows that one can choose $z \in \left\{x \in \Atplus, \theta \circ \phi_r^{-n}(x) \in \mathcal{O}_{L_{n+k_0}} \textrm{ for all } n \geq 1 \right\}$ such that $\phi_r(z) = S(z)$ with $S(T) \in T\cdot\mathcal{O}_L[\![T]\!]$, and we will assume in what follows that such a choice has been made.

\begin{lemm}
\label{lemm H_g(w)}
Assume that there exists $m_0 \geq 0$ such that for all $m \geq m_0$, the extension $L_m/L_{m_0}$ is Galois. Then the ring $\mathcal{O}_L[\![z]\!]$ is stable under the action of $\Gal(K_\infty/L_{m_0})$, and if $g \in \Gal(K_\infty/L_{m_0})$, there exists a power series $H_g(T) \in \mathcal{O}_L[\![T]\!]$ such that $g(z) = H_g(z)$. 
\end{lemm}
\begin{proof}
Let $f_0 = \max(m_0,k_0)$. Since for all $m \geq m_0$, $L_m/L_{m_0}$ is Galois, the set
$$\left\{x \in \Atplus, \theta \circ \phi_r^{-n}(x) \in \mathcal{O}_{L_{n+f_0}} \textrm{ for all } n \geq 1 \right\}$$
is stable under the action of $\Gal(K_\infty/L_{m_0})$, and by proposition \ref{existz}, this set is equal to $\mathcal{O}_L[\![\phi_r^{k_0-f_0}(z)]\!]$. In particular, if $g \in \Gal(K_\infty/L_{m_0})$, then $g(\phi_r^{k_0-f_0}(z)) \in \mathcal{O}_L[\![\phi_r^{k_0-f_0}(z)]\!]$ and so there exists $H_g(T) \in \mathcal{O}_L[\![T]\!]$ such that $H_g(\phi_r^{k_0-f_0}(z)) = g(\phi_r^{k_0-f_0}(z))$, and thus $H_g(z)=g(z)$. 
\end{proof}

\section{$p$-adic Hodge theory}
Let us assume that there exists $m_0 \geq 0$ such that for all $m \geq m_0$, the extension $L_m/L_{m_0}$ is Galois. Lemma \ref{lemm H_g(w)} shows that in this case we are in the exact same spot as the situation after lemma 5.15 of \cite{poyeton2022formal}. In particular, the exact same techniques apply.

We keep the notations from $\S 4$ and we let $\kappa : \Gal(K_\infty/L_{m_0}) \ra \OO_L^\times$ denote the character $g \mapsto H_g'(0)$. 

\begin{prop}
The character $\kappa: \Gal(K_\infty/L_{m_0}) \ra \OO_L^\times$ is injective and crystalline with nonnegative weights. 
\end{prop} 
\begin{proof}
This is the same as corollary 5.17 and proposition 5.19 of \cite{poyeton2022formal}.
\end{proof}

For $\lambda$ a uniformizer of $L_{m_0}$, let $(L_{m_0})_{\lambda}$ be the extension of $L_{m_0}$ attached to $\lambda$ by local class field theory. This extension is generated by the torsion points of a Lubin-Tate formal group defined over $L_{m_0}$ and attached to $\lambda$, and we write $\chi_{\lambda}^{L_{m_0}}~: \Gal((L_{m_0})_{\lambda}/L_{m_0}) \to \mathcal{O}_{L_{m_0}}^{\times}$ the corresponding Lubin-Tate character. Since $K_{\infty}/L_{m_0}$ is abelian and totally ramified, there exists $\lambda$ a uniformizer of $\mathcal{O}_{L_{m_0}}$ such that $K_{\infty} \subset (L_{m_0})_{\lambda}$.

\begin{prop}
\label{prop actually subset of L}
There exists $F \subset L$ and $r \geq 1$ such that $\kappa = N_{L_{m_0}/F}(\chi_{\lambda}^{L_{m_0}})^r$.
\end{prop}
\begin{proof}
Theorem 5.27 of \cite{poyeton2022formal} shows that there exists $F \subset L_{m_0}$ and $r \geq 1$ such that $\kappa = N_{L_{m_0}/F}(\chi_{\lambda}^{L_{m_0}})^r$. The fact that $\kappa$ takes its values in $\OO_L^\times$ shows that $F$ is actually a subfield of $L$. 
\end{proof}

Recall that relative Lubin-Tate groups are a generalization of usual formal Lubin-Tate group given by de Shalit in \cite{de1985relative}.  

\begin{theo}
\label{theo there is a formal group}
There exists $F \subset L$ and $r \geq 1$ such that $\kappa = N_{L/F}(\chi_{\lambda}^L)^r$. Moreover, there exists a relative Lubin-Tate group $S$, relative to the extension $F^{\mathrm{unr}} \cap L$ of $F$, such that if $L_\infty^S$ is the extension of $L$ generated by the torsion points of $S$, then $L_\infty \subset L_\infty^S$ and $L_\infty^S/L_\infty$ is a finite extension. 
\end{theo}
\begin{proof}
This is the same as \cite[Thm. 5.28]{poyeton2022formal} using proposition \ref{prop actually subset of L} instead of theorem 5.27 of ibid.
\end{proof}

\section{Isogenies}
By theorem \ref{theo there is a formal group}, there exists $F \subset L$ and a relative Lubin-Tate group $S$, relative to the extension $F^{\mathrm{unr}} \cap L$ of $F$, such that if $L_\infty^S$ is the extension of $L$ generated by the torsion points of $S$, then $L_\infty \subset L_\infty^S$ and $L_\infty^S/L_\infty$ is a finite extension. 

Let $\alpha$ be an element of $F^{\mathrm{unr}} \cap L$ such that $L_\infty^S$ is the field cut out by $<\alpha>$ of $F^{\mathrm{ab}}$ by local class field theory, so that the relative Lubin-Tate group $S$ is attached to $\alpha$. Up to replacing $L$ by a finite extension, we can assume that $L_\infty^S=L_\infty$ and we do so in what follows. We let $u_0 = 0$ and let $(u_n)_{n \in \N}$ be a nontrivial compatible sequence of roots of iterates of $[\alpha]$, the endomorphism of $S$ corresponding to the multiplication by $\alpha$, so that $[\alpha](u_{n+1})=u_n$ with $u_1 \neq 0$. We let $q$ denote the cardinal of the residue field of $F^{\mathrm{unr}} \cap L$ so that $\mathrm{wideg}([\alpha]) = q$. Let $\overline{u} = (u_0,\ldots) \in \Etplus$. By \S 9.2 of \cite{Col02}, there exists $u \in \Atplus$ whose image in $\Etplus$ is $\overline{u}$ and such that $\phi_q(u)=[\alpha](u)$, $g(w) = [\chi_\alpha(g)](u)$ for $g \in \G_L$. 

Recall that Cais and Davis have defined a ``canonical ring'' attached to $L_\infty/L$, denoted by $\A_{L_\infty/L}^+$ which is a subring of $\Atplus$ and is defined \textit{via} the tower of elementary extensions attached to $L_\infty/L$ by ramification theory. The following lemma shows that this canonical ring is related to the ring $\OO_L[\![u]\!]$ for the extension $L_\infty/L$:

\begin{lemm}
\label{lemm canonical ring equals www}
There exists $k \geq 0$ such that $\A_{L_\infty/L}^+ = \phi_q^{-k}(\OO_L[\![u]\!])$.
\end{lemm}
\begin{proof}
See \cite[Lemm. 8.1]{poyeton2022formal}. Be mindful that here $L$ and $u$ play respectively the role of $E$ and $w$ in ibid.
\end{proof}

Recall that $(w_n)_{n \in \N}$ is a $Q$-consistent sequence, where $Q$ commutes with $P$ and is such that $Q(T) = T^s \mod \mathfrak{m}_L$, and that $w \in \Atplus$ is such that $\theta \circ \phi_r^{-n}(w) = w_n$. 

\begin{prop}
\label{prop almost canonical lift}
There exists $i \geq 0$ such that $\phi_r^i(w) \in \A_{L_\infty/L}^+$.
\end{prop}
\begin{proof}
The proof is exactly the same as in \cite[Prop. 8.2]{poyeton2022formal}.
\end{proof}

\begin{prop}
\label{prop exists iso}
There exists $d \geq 1$ such that there is an isogeny from $[\alpha^d]$ to $Q$. 
\end{prop}
\begin{proof}
Lemma \ref{lemm canonical ring equals www} and proposition \ref{prop almost canonical lift} show that there exist $i \geq 0$ and $h(T) \in \OO_L[\![T]\!]$ such that $w=h(\phi_r^{-i}(u))$. Let $d$ be such that $\phi_r = \phi_q^{\circ d}$ and let $\tilde{u} = \phi_r^{-i}(u)$, so that $w=h(\tilde{u})$. For $g \in \G_L$, we have $\phi_r(w)=Q(w)$ so that $Q(w)=\phi_r(w) = \phi_r(h(\tilde(u)) = h(\phi_r(\tilde{u}))$ and thus $Q \circ h(\tilde{u}) = h \circ [\alpha^d](\tilde{u})$ which means that $Q \circ h = h \circ [\alpha^d]$. 
\end{proof}

\begin{theo}
\label{theo maintheo}
Let $(P,U)$ be a couple of power series in $T\cdot \OO_K[\![T]\!]$ such that $P \circ U = U \circ P$, with $P'(0) \in \mathfrak{m}_K$ and $U'(0) \in \OO_K^\times$, and we assume that $P(T) \neq 0 \mod \mathfrak{m}_K$ and that $U'(0)$ is not a root of unity. Then there exists a finite extension $E$ of $K$, a Lubin-Tate formal group $S$ defined over $\OO_L$, where $E/L$ is a finite extension, endomorphisms of this formal group $P_S$ and $U_S$ over $\OO_E$, and a power series $h(T) \in T\OO_E[\![T]\!]$ such that $P \circ h = h \circ P_S$ and $U \circ h = h \circ U_S$ if and only if the following two conditions are satisfied:
\begin{enumerate}
\item there exists $V \in T\cdot\OO_K[\![T]\!]$, commuting with $P$, and an integer $d \geq 1$ such that $Q(T)=T^{p^d} \mod \mathfrak{m}_K$ where $Q=V \circ P$ ;
\item there exists a finite extension $E$ of $K$ and a sequence $(\alpha_n)_{n \in \N}$ where $\alpha_0 \neq 0$ is a root of $Q$ and $Q(\alpha_{n+1})=\alpha_n$ such that for all $n \geq 1$, the extension $E(\alpha_n)/E$ is Galois.
\end{enumerate}
\end{theo}
\begin{proof}
Lemmas \ref{lemma renormPintoQ} and \ref{lemma endo Lubin implies Q-consistent} of \S 2 imply that if such a Lubin-Tate formal group exist then the two conditions are satisfied. 

If those two conditions are satisfied, then proposition \ref{prop exists iso} shows that there exists a finite extension $E$ of $K$, a subfield $F$ of $E$, a relative Lubin-Tate group $S$, relative to the extension $F^{\mathrm{unr}} \cap E$ of $F$, and an endomorphism $Q_S$ of $S$ such that there exists an isogeny from $Q_S$ to $Q$. Thus there exists an isogeny from an endomorphism $P_S$ of $S$ to $P$. In order to conclude, it suffices to notice that a relative Lubin-Tate formal group $S$, relative to an extension $F^{\mathrm{unr}} \cap E$ of $F$ is actually isomorphic over $F^{\mathrm{unr}} \cap E$ to a Lubin-Tate formal group $S'$ defined over $F$.
\end{proof}

\section{A particular case of Lubin's conjecture}
We now apply the results from the previous sections to the particular case where $P(T) = T^p \mod \mathfrak{m}_K$. Let $P, U \in T\cdot \OO_K[\![T]\!]$ such that $P \circ U = U \circ P$, with $P(T) = T^p \mod \mathfrak{m}_K$ and $U'(0) \in \OO_K^\times$ not a root of unity. We consider as in \S 3 a $P$-consistent sequence $(v_n)$ and we let $K_n = K(v_n)$ for $n \geq 0$. We let $n_0$ be as in proposition \ref{prop K(vn) tot ramified}. 

\begin{prop}
There exists $m_0 \geq 0$ such that for all $m \geq m_0$, the extension $K_m/K_{m_0}$ is Galois.
\end{prop}
\begin{proof}
By \cite[Prop. 3.2]{lubin1994nonarchimedean}, the roots of iterates of $P$ are exactly the fixed points of the iterates of $U$. Up to replacing $U$ by some power of $U$, we can assume that $U'(0) = 1 \mod \mathfrak{m}_K$ and that there exists $n \geq n_0$ such that $U(v_n) = v_n$ but $U(v_{n+1}) \neq v_{n+1}$ (since $U(T)-T$ admits only a finite number of roots in the unit disk). 

Since $U(v_n) = v_n$ and $U$ commutes with $P$, this implies that $U(v_{n+1})$ is also a root of $P(T)-v_n$. The discussion on page 333 of \cite{lubin1994nonarchimedean} shows that the set $\{U^{\circ k}(v_{n+1})\}_{k \in \N}$ has cardinality a power of $p$, and is not of cardinal $1$ since $U(v_{n+1}) \neq v_{n+1}$ by assumption. Since $P(T)-v_n$ has exactly $p$ roots, this implies that the set $\{U(v_{n+1})\}$ has cardinality $p$, and thus all the roots of $P(T)-v_n$ are contained in $K_{n+1}$, so that $K_{n+1}/K_n$ is Galois. 

Let $m > n$. The extension $K_m/K_n$ is generated by all the roots of $P^{\circ (m-n)}(T)-v_n = P^{\circ (m-n)}(T)-U(v_n)$. Since $U$ swaps all the roots of $P(T)-v_n$, it is easy to see that the $U$-orbit $\{U^{\circ k}(v_m)\}_{k \geq 0}$ contains all the roots of $P^{\circ (m-n)}(T)-v_n$, so that $K_m/K_n$ is Galois. This prove the proposition.
\end{proof}

We are now in the conditions of our theorem \ref{theo maintheo}, which yields the following:

\begin{coro}
Lubin's conjecture is true for $(P,U)$.
\end{coro}

\bibliographystyle{amsalpha}
\bibliography{bibli}

\providecommand{\bysame}{\leavevmode\hbox to3em{\hrulefill}\thinspace}
\providecommand{\MR}{\relax\ifhmode\unskip\space\fi MR }
\providecommand{\MRhref}[2]{%
  \href{http://www.ams.org/mathscinet-getitem?mr=#1}{#2}
}
\providecommand{\href}[2]{#2}
\begin{thebibliography}{CDL16}

\bibitem[BC09]{BrinonConrad}
Olivier Brinon and Brian Conrad, \emph{{CMI} summer school notes on $p$-adic
  hodge theory}.

\bibitem[Ber16a]{Ber14iterate}
Laurent Berger, \emph{Iterated extensions and relative {L}ubin-{T}ate groups},
  Ann. Math. Qu\'e. \textbf{40} (2016), no.~1, 17--28.

\bibitem[Ber16b]{Ber14MultiLa}
\bysame, \emph{Multivariable {$(\varphi,\Gamma)$}-modules and locally analytic
  vectors}, Duke Math. J. \textbf{165} (2016), no.~18, 3567--3595.

\bibitem[Ber17]{Ber17}
\bysame, \emph{Lubin's conjecture for full {$p$}-adic dynamical systems},
  Publications math\'ematiques de {B}esan\c{c}on. {A}lg\`ebre et th\'eorie des
  nombres, 2016, Publ. Math. Besan\c{c}on Alg\`ebre Th\'eorie Nr., vol. 2016,
  Presses Univ. Franche-Comt\'e, Besan\c{c}on, 2017, pp.~19--24.

\bibitem[Ber19]{berger2019nonarchimedean}
\bysame, \emph{Nonarchimedean dynamical systems and formal groups}, Proceedings
  of the American Mathematical Society \textbf{147} (2019), no.~4, 1413--1419.

\bibitem[CC98]{cherbonnier1998representations}
Fr{\'e}d{\'e}ric Cherbonnier and Pierre Colmez, \emph{Repr{\'e}sentations
  p-adiques surconvergentes}, Inventiones mathematicae \textbf{133} (1998),
  no.~3, 581--611.

\bibitem[CD15]{cais2015canonical}
Bryden Cais and Christopher Davis, \emph{Canonical {C}ohen rings for norm
  fields}, International Mathematics Research Notices \textbf{2015} (2015),
  no.~14, 5473--5517.

\bibitem[CDL16]{cais2014characterization}
Bryden Cais, Christopher Davis, and Jonathan Lubin, \emph{A characterization of
  strictly {APF} extensions}, J. Th\'eor. Nombres Bordeaux \textbf{28} (2016),
  no.~2, 417--430.

\bibitem[Col02]{Col02}
Pierre Colmez, \emph{Espaces de {B}anach de dimension finie}, Journal of the
  Institute of Mathematics of Jussieu \textbf{1} (2002), no.~3, 331--439.

\bibitem[dS85]{de1985relative}
Ehud de~Shalit, \emph{Relative {L}ubin-{T}ate {G}roups}, Proceedings of the
  American Mathematical Society \textbf{95} (1985), no.~1, 1--4.

\bibitem[Fon94]{fontaine1994corps}
Jean-Marc Fontaine, \emph{Le corps des p{\'e}riodes $p$-adiques},
  Ast{\'e}risque (1994), no.~223, 59--102.

\bibitem[Li96]{Li96}
Hua-Chieh Li, \emph{When is a $p$-adic power series an endomorphism of a formal
  group?}, Proceedings of the American Mathematical Society \textbf{124}
  (1996), no.~8, 2325--2329.

\bibitem[Li97a]{Li97}
\bysame, \emph{Isogenies between dynamics of formal groups}, journal of number
  theory \textbf{62} (1997), no.~2, 284--297.

\bibitem[Li97b]{Li97b}
\bysame, \emph{$p$-adic power series which commute under composition},
  Transactions of the American Mathematical Society \textbf{349} (1997), no.~4,
  1437--1446.

\bibitem[LMS02]{laubie2002systemes}
Fran{\c{c}}ois Laubie, Abbas Movahhedi, and Alain Salinier, \emph{Syst{\`e}mes
  dynamiques non archim{\'e}diens et corps des normes ({N}on-{A}rchimedean
  {D}ynamic {S}ystems and {F}ields of {N}orms)}, Compositio Mathematica
  \textbf{132} (2002), no.~1, 57--98.

\bibitem[Lub94]{lubin1994nonarchimedean}
Jonathan Lubin, \emph{Nonarchimedean dynamical systems}, Compositio Mathematica
  \textbf{94} (1994), no.~3, 321--346.

\bibitem[Poy22]{poyeton2022formal}
L{\'e}o Poyeton, \emph{Formal groups and lifts of the field of norms}, Algebra
  \& Number Theory \textbf{16} (2022), no.~2, 261--290.

\bibitem[Sar05]{sarkis2005lifting}
Ghassan Sarkis, \emph{On lifting commutative dynamical systems}, Journal of
  Algebra \textbf{293} (2005), no.~1, 130--154.

\bibitem[Sar10]{sarkis10}
\bysame, \emph{Height-one commuting power series over {$\Bbb Z_p$}}, Bull.
  Lond. Math. Soc. \textbf{42} (2010), no.~3, 381--387.

\bibitem[Spe18]{specter2018crystalline}
Joel Specter, \emph{The crystalline period of a height one $p$-adic dynamical
  system}, Transactions of the American Mathematical Society \textbf{370}
  (2018), no.~5, 3591--3608.

\bibitem[SS13]{sarkis2013galois}
Ghassan Sarkis and Joel Specter, \emph{Galois extensions of height-one
  commuting dynamical systems}, Journal de th{\'e}orie des nombres de Bordeaux
  \textbf{25} (2013), no.~1, 163--178.

\bibitem[Win83]{Win83}
Jean-Pierre Wintenberger, \emph{Le corps des normes de certaines extensions
  infinies de corps locaux; applications}, Annales scientifiques de l'Ecole
  Normale Sup{\'e}rieure, vol.~16, Soci{\'e}t{\'e} math{\'e}matique de France,
  1983, pp.~59--89.

\end{thebibliography}
\end{document}